\documentclass[11pt,final]{article}
\usepackage{amsmath}
\usepackage{amssymb}
\usepackage{latexsym}
\usepackage{amsthm}
\usepackage{fullpage}
\usepackage{graphicx}
\usepackage[colorinlistoftodos]{todonotes}
\usepackage{marginnote}
\usepackage{hyperref}
 \numberwithin{equation}{section}
\usepackage[notcite, notref]{showkeys}
\usepackage{caption}
\usepackage{subcaption}
\usepackage[mathscr]{eucal}

\usepackage{mathrsfs}
\usepackage{wrapfig}
\newtheorem{theorem}{Theorem}[section]
\newtheorem{lemma}[theorem]{Lemma}

\theoremstyle{remark}
\newtheorem{remark}[theorem]{Remark}

\theoremstyle{definition}

\newcommand{\R}{{\mathbb R}}
\newcommand{\N}{{\mathbb N}}

\newcommand{\be}{\begin{eqnarray}}

\newcommand{\ee}{\end{eqnarray}}

\newcommand{\md}{{\rm d}}

\renewcommand{\O}{\Omega}

\newcommand{\F}{{%
F}}

\renewcommand{\det}{{\rm det}\,}

\newcommand{\vstress}{{S}}

\def\argmin{\mathop{\rm argmin}}
\def\sm{\mathop{\rm sym}}

\def\curl{\mathop{\rm curl}}
\def\SO{\mathrm{SO}}
\def\GL{\mathrm{GL}}
\def\Skew{\mathrm{Skew}}
\def\F{\mathcal F}

\def\argmin{\mathop{\rm argmin}}

\def\Id{\mathbf{Id}}

\def\eps{\varepsilon}

\def\XXint#1#2#3{{\setbox0=\hbox{$#1{#2#3}{\int}$}
     \vcenter{\hbox{$#2#3$}}\kern-.5\wd0}}

\newcommand{\EEE}{\color{black}}

\author{Patrick Dondl\footnote{Department for Applied Mathematics, University of Freiburg, Germany
 } \  \ Martin Jesenko\footnote{University of Ljubljana, Faculty of Civil and Geodetic Engineering, Jamova cesta 2, 1001 Ljubljana, Slovenia}\ \ \ Martin Kru\v{z}\'{i}k\footnote{Institute of Information Theory and Automation,
Czech Academy of Sciences, Pod vod\'{a}renskou
v\v{e}\v{z}\'{\i}~4, CZ-182~08~Praha~8, Czechia (corresponding
address) \& Faculty of Civil Engineering, Czech Technical
University, Th\'{a}kurova 7, CZ-166~ 29~Praha~6, Czechia
}\ \ \ Jan Valdman\footnote{Institute of Information Theory and Automation,
Czech Academy of Sciences, Pod vod\'{a}renskou
v\v{e}\v{z}\'{\i}~4, CZ-182~08~Praha~8, Czechia \&  Department of Mathematics, Faculty of Science,
University of South Bohemia,
37005~\v{C}esk\'{e}~Bud\v{e}jovice, Czechia }
}
\title{Linearization and Computation for Large-Strain Viscoelasticity}
\begin{document}
\maketitle
\begin{abstract}
Time-discrete numerical minimization schemes for simple viscoelastic materials in the large strain Kelvin-Voigt rheology are not well-posed due to non-quasiconvexity of the dissipation functional. A possible solution is to  resort into non-simple material models with higher-order gradients of deformations. This makes, however,  numerical computations much more involved. Here we propose another approach relying on  local minimizers of the simple-material model. Computational tests are provided showing a very good agreement between our model and the original one.

\end{abstract}

\section{Introduction}
In-time-semidiscretized problems of nonlinear viscoelasticity in  Kelvin's-Voigt's rheology  lead to a sequence of minimization problems where a current solution depends on the previous one. More precisely, starting from an initial condition $y^0$  we must find for $k=1,\ldots T/\tau$
\begin{align}\label{eq:intro-minimization}
   y^k\in \argmin_{y} \int_\Omega W(\nabla y)\, dx + \frac1{2\tau} D^2(\nabla y,\nabla y^{k-1})\ ,
\end{align}
where $\tau>0$ is a time step, $T$ is a final time, $W$ is an elastic energy density, and $D$ is a  dissipation due to viscosity. However, physically acceptable dissipation functions  are not (Morrey) quasiconvex \cite{Antman}. As we are interested in the limit for $\tau\to 0$  the problem \eqref{eq:intro-minimization} does not necessarily have a minimizer.  This suggests microstructure formation due to rapidly oscillating deformation gradients, which, however, is not observed in reality. The aim of the present work is to explore  this discrepancy.

Due to dissipative effects we expect  $y$ to be found in  a small  neighborhood of $y^{k-1}$ in \eqref{eq:intro-minimization} that leads to local minimization.  For that reason,  we linearize the energy functional in \eqref{eq:intro-minimization} in the spirit of \cite{DNP}. Thus we obtain a quadratic energy contribution stemming from the dissipation and the stress from the elastic energy.  The resulting  functional is not coercive,   but in spite of this, we show that  minimizers exist.

We  perform numerical experiments comparing the original nonlinear scheme in \eqref{eq:intro-minimization} with our  linearized scheme for small time steps  that show a very good quantitative  agreement. Moreover, both schemes   satisfy an energy balance.

\section{Modeling of nonlinear viscoelasticity}

If we neglect inertial effects,  the deformation $y:[0,T]\times \Omega\to\R^d$ of a  nonlinear viscoelastic material in Kelvin's-Voigt's rheology obeys the following system of equations
\begin{subequations}\label{syst:viscoel}
\begin{gather}
\label{eq:viscoel}
-{\rm div}\Big(\partial_FW(\nabla y)  + \partial_{\dot F}R(\nabla y, \nabla\dot y)  \Big) =  f\text{ in $  [0,T] \times  \Omega$,} \\
\Big(\partial_FW(\nabla y)   + \partial_{\dot F}R(\nabla y,\nabla\dot y) \Big)n=g\text{ on }  [0,T] \times\Gamma_N, \\
y=y_D \text{ on } [0,T] \times\Gamma_D, \\
y(0,\cdot)=y^0 \text{ in $ \Omega$,}
\end{gather}
\end{subequations}
Here, $[0,T]$ is a process time interval,   $\Omega\subset\R^d$  ($d=2$ or $d=3$)  is a smooth bounded domain representing the reference configuration, $\Gamma_D\cup\Gamma_N$ is a disjoint partition of $\partial\Omega$  such that $\Gamma_D$ has a positive $d-1$-dimensional Hausdorff  measure. Further, $f:[0,T]\times\Omega\to\R^d$ is a density of body forces and  $g:[0,T]\times\Gamma_N\to\R^d$ is a density of surface forces.  We also impose  boundary datum $ y_D : [0,T]\times \Gamma_D \to \R^d$ and an initial condition $y^0: \Omega \to \R^d$.

Material properties are described by the stored energy density  $W:\R^{d\times d}\to  [0,\infty]$, where its gradient represents the first Piola-Kirchhoff stress tensor, and by a (pseudo)potential of dissipative forces $R:  \R^{d \times d} \times \R^{d \times d} \to [0,\infty)  $.

The  stress tensor
$\vstress(F,\dot F):= \partial_{\dot F} R(F,\dot F)$  has its origin in viscous dissipative mechanisms of the material.   Naturally, we require that $R(F,\dot F)\ge R(F,0)=0$.  The viscous stress tensor must comply with the time-continuous frame-indifference principle  meaning that for all $F, \dot{F}\in \R^{d\times d}$
\begin{align*}
\vstress(F,\dot F)=F\tilde\vstress(C,\dot C)
\end{align*}
 with some $\tilde\vstress : \R^{d \times d}_{\sm} \times \R^{d \times d}_{\sm} \to \R^{d \times d}_{\sm} $  where $ C = F^{\top} F $ and $ \dot{C} = \dot{F}^{\top} F + F^{\top} \dot{F} $.   This condition constraints
$R$ so that (see \cite{Antman, Antman:04,MOS})
\begin{align}\label{eq:frame indifference-R}
R(F,\dot F)=\tilde R(C,\dot C)\ ,
\end{align}
for some nonnegative function $\tilde R$ such that $\tilde R(C,0)=0$.  In other words,  $R$ must depend on the right Cauchy-Green strain tensor $C$ and its time derivative $\dot C$.  On the other hand, this property makes the analysis of \eqref{syst:viscoel} very difficult. To our best knowledge,  the existence of a solution has not been established up to now.  On the other hand, if the stored energy is enriched by the second gradient of the deformation, existence results are available in various settings, see  \cite{mfmk,KR} in three dimensions and \cite{BS,MOS} in one dimension.

Namely, a standard approach to obtain solutions for such evolution problems  is to introduce an implicit time discretization. For simplicity, we will  assume that $y_D$ is constant in time in the sequel.
Let $\tau>0$ be a time step such that $T/\tau\in \N$. Given  an initial condition $y^0_\tau$, for $k=1,\ldots, T/\tau$, we solve the following problem stemming from \eqref{syst:viscoel}.
\begin{alignat}{3}\label{eq:PDEvisco-formulation}
-{\rm div}\Big(\partial_FW(\nabla y^k_\tau)   + \partial_{\dot F}R(\nabla y_\tau^k, \nabla\frac{y^k_\tau-y^{k-1}_\tau}{\tau})  \Big) &=  f^k_\tau:=\tau^{-1}\int_{(k-1)\tau}^{k\tau}f(s,\cdot)\,\md s \quad \text{ in  $\Omega$,}
\nonumber
\\
\Big(\partial_FW(\nabla y^k_\tau)   + \partial_{\dot F}R(\nabla y_\tau^k, \nabla\frac{y^k_\tau-y^{k-1}_\tau}{\tau})\Big)n&= g^k_\tau:=\tau^{-1}\int_{(k-1)\tau}^{k\tau}g(s,\cdot)\,\md s  \quad \text{ on } \Gamma_N, \\
y^k_\tau&=y_D \quad \text{ on } \Gamma_D. \nonumber
\end{alignat}

\paragraph{Stored elastic energy and body forces:}  Assume that
 $W: \R^{d \times d} \to [0,\infty]$ is
\begin{align}\label{assumptions-W}
\begin{split}
(i)& \ \ W \text{ is smooth on matrices with positive determinants},\\
(ii)& \ \  W(QF) = W(F) \text{ for all } F \in \R^{d \times d}, Q \in \SO(d),\\
(iii)& \ \ W(F)\ge c(-1+|F|^p) \text{ for some $c>0$, $p>1$,  and  for all $F\in\R^{d\times d}$}
\end{split}
\end{align}
where $\SO(d) = \lbrace Q\in \R^{d \times d}: Q^\top Q = \Id, \, \det Q=1 \rbrace$ are proper rotations.

We introduce the nonlinear elastic energy $\phi^k_\tau: W^{1,p}(\Omega;\R^d) \to [0,\infty]$ by
\begin{align}\label{nonlinear energy}
\phi^k_\tau(y) = \int_\Omega W(\nabla y(x))\,\md x   - \int_\Omega f^k_\tau(x)\cdot y(x)\,\md x-\int_{\Gamma_N}g^k_\tau(x)\cdot y(x)\,\md S
\end{align}
for  $y: W^{1,p}(\Omega;\R^d) \to \R^d$.  \EEE

\paragraph{Dissipation potential and viscous stress:} Consider a time dependent deformation $y: [0,T] \times \Omega \to \R^d$. Viscosity is not only related to the strain $\nabla y(t,x)$, but also to the strain rate $ \nabla \dot y(t,x)$ and can be expressed in terms of a  dissipation potential $R(\nabla y, \nabla \dot y)$.  An admissible potential has to satisfy frame indifference in the sense (see \cite{Antman, MOS})
\begin{align}\label{R: frame indiff}
R(F,\dot{F}) = R(QF,Q(\dot{F} + AF))  \ \ \  \forall  Q \in \SO(d),\, A \in \Skew(d)
\end{align}
for all $F \in \GL_+(d)$ and $\dot{F} \in \R^{d \times d}$, where
$$\GL_+(d) = \lbrace F \in \R^{d \times d}: \det F>0 \rbrace, \quad \Skew(d) = \lbrace A  \in \R^{d \times d}: A=-A^\top \rbrace.$$

Following the discussion in \cite[Section 2.2]{MOS}, from the point of modeling  it is much more  convenient to postulate the existence of a (smooth) global distance $D: \GL_+(d) \times \GL_+(d) \to [0,\infty)$ satisfying $D(F,F) = 0$ for all $F \in \GL_+(d)$, from which an associated dissipation potential $R$ can be calculated by
\begin{align}\label{intro:R}
R(F,\dot{F}) := \lim_{\eps \to 0} \frac{1}{2\eps^2} D^2(F+\eps\dot{F},F) = \frac{1}{4} \partial^2_{F_1^2} D^2(F,F) [\dot{F},\dot{F}]
\end{align}
for $F \in \GL_+(d)$, $\dot{F} \in \R^{d \times d}$,  where $\partial^2_{F_1^2} D^2(F_1,F_2)$ denotes the Hessian of  $D^2$ in direction of $F_1$ at $(F_1,F_2)$, being a fourth order tensor.
\\
For the sake of simplicity, in the present work we use
\[ D(F_1,F_2)=|C_1-C_2|, \]
where $C_i=F_i^\top F_i$ for $i=1,2$.
In this case, a direct computation yields
\begin{equation} \label{eq:R-concrete}
    R(F,\dot{F}) = 2 | \sm ( F^{\top}  \dot{F} ) |^{2}
\end{equation}

For justification of this choice for $D$ and for further  examples of admissible dissipation distances, we refer the reader to \cite[Section 2.3]{MOS}. %

Now,  solutions to  \eqref{eq:PDEvisco-formulation} can \emph{formally} be written as  solutions of the following sequence of problems:
\begin{equation} \label{eq:visco_minimization_naive}
\begin{gathered}
\text{For a given  $y^{k-1}_{\tau}\in W^{1,p}(\O;\R^d) $ with $ y^{k- 1}_\tau=y_D $ on $ \Gamma_D$} \\
\text{ minimize } \quad
\phi^k_\tau(y)  +\frac{1}{2\tau}\int_\O D^2\Big(\nabla y,\nabla y^{k-1}_{\tau}\Big)\,\md x  \\
\text{ subject to } \quad y\in W^{1,p}(\O;\R^d), \ y=y_D \text{ on } \Gamma_D.
\end{gathered}
\end{equation}

Due to the non-quasiconvexity of $F\mapsto D(F,G)$ for a given $G$, see \cite{dacorogna}, it is immediately clear that one cannot expect existence of minimizers due to the formation of microstructure. We note, however, that such microstructure is usually not  observed in  viscoelastic  materials. Our conjecture is that the issue here lies with the fact that \eqref{eq:visco_minimization_naive} is a global minimization, whence a viscous evolution should only explore the energy and dissipation in a small neighborhood of the current state. We therefore propose to minimize instead  in each time step a problem that is linearized around $y^{k-1}_{\tau}$ in a suitable sense.

\section{Linear viscoelasticity}\label{sec:linvis}
To motivate our linearization approach for this implicit time-discretization problem, we first consider the case of (already) linearized viscoelasticity. Here, we simply recover the original evolution equations, which (neglecting any forces) are given by the linear Kelvin-Voigt viscoelasticity problem
\begin{alignat}{3}\label{eq:lin_evol}
 -{\rm div} (\mathbb{C}\nabla u+\mathbb{D}\nabla \dot u  ) &= 0 \qquad \text{in $[0,T]\times\Omega$}, \nonumber \\
u&=u_D \,\quad \text{on $\Gamma_D$}, \\
u(0,\cdot)&=u^0 \,\,\quad \text{in $\Omega$}\nonumber
\end{alignat}
Here $\mathbb{C}$ and $\mathbb{D}$ are positive definite fourth-order tensors of elasticity and viscosity  coefficients, respectively.
Note that in this linearized setting, we work with a displacement field $u$ instead of the deformation $y$. %

The variational formulation for the linear counterpart of \eqref{eq:PDEvisco-formulation} reads:
For a given  $u^{k-1}_{\tau}\in W^{1,2}(\O;\R^d) $ with $ u^{k-1}_\tau=u_D $ on $ \Gamma_D$, find the minimizer $u^{k}_{\tau}$ of
\begin{align}\label{min:linvis}
u \mapsto \frac{1}{2}\int_\O\mathbb{C}\nabla u : \nabla u \,\md x+\frac{1}{2\tau}\int_\O \mathbb{D}(\nabla u- \nabla u^{k-1}_{\tau}) : (\nabla u- \nabla u^{k-1}_{\tau})\,\md x.
\end{align}
under the constraints $ u\in W^{1,2}(\O;\R^d) $ and $ u=u_D $ on $ \Gamma_D$.
Since
the functional is quadratic and
convex in $ \nabla u $, the minimizer indeed exists and is unique. It is clear that this iterative minimization scheme, in the limit of small time steps $\tau$, leads to a solution of \eqref{eq:lin_evol}.

Our idea is to reformulate the minimization problem to \eqref{min:linvis} in terms of a minimizing increment in displacement. More precisely, we write
$ u = u_\tau^{k-1} + \tau v $
and minimize
\begin{equation}
v \mapsto \frac{1}{2}\int_\O
\left(
\mathbb{C}\nabla u^{k-1}_\tau : \nabla u^{k-1}_\tau + 2\tau\mathbb{C}\nabla u^{k-1}_\tau : \nabla v +\tau^2 \mathbb{C}\nabla v : \nabla v + \tau\mathbb{D}\nabla v : \nabla v
\right)\,\md x,
\end{equation}
or, equivalently, after subtracting the constant first term and rescaling by $\frac{1}{\tau}$, minimize
\begin{equation} \label{eq:lin-functional-v}
v \mapsto \frac{1}{2}\int_\O
\left( 2\mathbb{C}\nabla u^{k-1}_\tau : \nabla v +\tau \mathbb{C}\nabla v : \nabla v + \mathbb{D}\nabla v : \nabla v
\right)
\,\md x
\end{equation}
under the constraints $ v \in W^{1,2}(\O;\R^d) $ and $ v=0 $ on $ \Gamma_D$.
We observe that the second term is dominated by the third for small $\tau$.
Therefore we minimize
\begin{equation} \label{eq:lin-lin-functional-v}
v \mapsto \frac{1}{2}\int_\O
\left(
2\mathbb{C}\nabla u^{k-1}_\tau : \nabla v  +\mathbb{D}\nabla v : \nabla v
\right)
\,\md x
\end{equation}
instead.
This problem admits a unique minimizer $v$ which solves the associated Euler-Lagrange equation
\begin{equation}\label{eq:EL-lin-lin}
    -{\rm div}(\mathbb{C}\nabla u^{k-1}_\tau + \mathbb{D}\nabla v ) = 0. \end{equation}
Then $ u^k = u_\tau^{k-1} + \tau v $ fulfils
\begin{equation}\label{eq:EL-lin-lin-step}
    -{\rm div}(\mathbb{C}\nabla u^{k-1}_\tau + \frac{1}{\tau}\mathbb{D}\nabla(  u^k - u_\tau^{k-1})) = 0.
\end{equation}
We note that this is an explicit Euler time-step for the original evolution equation \eqref{eq:lin_evol}, wheres the standard minimizing movement results in the implicit scheme
\begin{equation}\label{eq:EL-lin-step}
    -{\rm div}(\mathbb{C}\nabla u^{k}_\tau + \frac{1}{\tau}\mathbb{D}\nabla( u^k_\tau - u_\tau^{k-1})) = 0.
\end{equation}
Note that we obtained this solution from the linearized functional \eqref{eq:lin-lin-functional-v} and not from \eqref{eq:lin-functional-v}. This suggests that we can write a minimization problem in the velocity $v$ for the dissipated power -- and still obtain a reasonable time-stepping scheme for linearized Kelvin-Voigt viscoelasticity.

Emboldened by this result for the linearized problem -- for which existence of solutions is of course already well known -- we attempt to treat the fully non-linear case.

\section{The limit for the nonlinear visco-elastic minimization problem}

We now consider again the fully nonlinear implicit time-stepping problem for visco-elasticity.

\subsection{Linearization of the minimization problem}
Starting with the time-step minimization problem
\begin{equation}
y_\tau^{k} = \argmin_{w} \int_{\Omega}
\left(
W( \nabla w(x) ) + \frac{1}{2 \tau} D^2( \nabla w(x) , \nabla y_\tau^{k-1}(x) ) \right)
\,\md x
\end{equation}
(where $ D(F,G) = | F^{\top} F - G^{\top} G | $). Now we rewrite this minimization problem by substituting $ y_\tau^{k} = y_\tau^{k-1} + \tau z_{k} $. Then
\begin{equation}
z_{k} = \argmin_{z} {\cal F}_{y_\tau^{k-1},\tau}(z), \end{equation}
where
\begin{equation}
 {\cal F}_{y,\tau}(z) := \frac{1}{\tau} \int_{\Omega} \left( W( \nabla y(x) + \tau \nabla z(x) ) - W( \nabla y(x)) + \frac{1}{2 \tau} D^2( \nabla y(x) + \tau \nabla z(x) , \nabla y(x) ) \right)
 \,\md x
\end{equation}
denotes the power.
For any given previous state, this is now a singularly perturbed functional, so we explore the limit $\lim_{\tau \to 0} {\cal F}_{y,\tau} $. Let us write for the sake of simplicity $ Y(x) := \nabla y(x) $ and $ Z(x) := \nabla z(x) $.
As for the first two terms, we may assess
\begin{equation}
\frac{ W( Y(x) + \tau Z(x) ) - W( Y(x)) }{\tau}
= \partial_F W( Y(x) ) : Z(x) + o( \tau )
\end{equation}
If we assume $ y \in W^{1,\infty}( \Omega ; \R^{d} ) $, then $ | \nabla \partial_F W( \nabla y ) | \le M $.
Since for every $x$ there exists $ \xi(x) \in [0,\tau] $ such that
\begin{equation}
W( Y(x) + \tau Z(x) ) - W( Y(x) )
= \partial_F W( Y(x) ) : \tau Z(x)
+ \nabla \partial_F W( Y(x) ) \xi(x) Z(x) \cdot  \xi(x) Z(x),
\end{equation}
we get
\begin{equation}
\frac{ W( Y(x) + \tau Z(x) ) - W( Y(x)) }{\tau} - \partial_F W( Y(x) ) : Z(x)
= \frac{ \xi(x)^{2} }{\tau}  \partial^2_{F^2} W( Y(x) ) Z(x) \cdot Z(x)
\end{equation}
and
\begin{equation}
\left| \frac{ W( Y(x) + \tau Z(x) ) - W( Y(x)) }{\tau} - \partial_F W( Y(x) ) : Z(x) \right|
\le \tau M | Z(x) |^{2}
\end{equation}

For the second one,
\begin{align*}
\frac{1}{ 2 \tau^{2} } D( Y + \tau Z , Y )^{2} %
& = \frac{1}{ 2 \tau^{2} } | ( Y + \tau Z )^{\top} ( Y + \tau Z ) - Y^{\top} Y |^{2} \\
& = \frac{1}{2} | Z^{\top} Y + Y^{\top} Z + \tau Z^{\top} Z |^{2} \\
& = \frac{1}{2} | Z^{\top} Y + Y^{\top} Z |^{2} + \tau ( Z^{\top} Y + Y^{\top} Z ) : ( Z^{\top} Z ) + \frac{ \tau^{2} }{2} | Z^{\top} Z |^{2} \\
& = 2 | \sm( Y^{\top} Z ) |^{2} + 2 \tau \ Y : Z Z^{\top} Z + \frac{ \tau^{2} }{2} | Z^{\top} Z |^{2}
\end{align*}
Therefore, for fixed $y$ and $z$ we expect that
\begin{equation} \label{eq:gam_limit_min_prob}
\lim_{\tau \to 0} \F_{y,\tau}(z)
= \int_{\Omega}  \left( \partial_F W( \nabla y(x) ) : \nabla z(x) + 2 | \sm( ( \nabla y(x) )^{\top} \nabla z(x) ) |^{2} \right)
\,\md x.
\end{equation}
Then, the solution to the viscoelastic evolution problem is given by
\begin{equation} \label{eq:step}
y^{k} = y^{k-1} + \hat{\tau} z
\end{equation}
for small $\hat{\tau}$, and $z$ calculated as the minimizer of \eqref{eq:gam_limit_min_prob} (with $y=y^{k-1}$) in each time step.

The heuristic interpretation of this method is that we derived a linearized functional (depending on given state) whose minimizer is the velocity in the evolution problem corresponding to the state. Equation \eqref{eq:step} then is nothing but a forward Euler time-step with this velocity.

\subsection{Existence of minimizers}
Let us explore properties of the functional from \eqref{eq:gam_limit_min_prob}.
For that reason, let us fix some $ y \in C^{2}( \overline{\Omega} ; \R^{3} ) $ such that $ \inf_{x \in \Omega} \det \nabla y(x) > 0 $, and denote
\[ \F_y(z) = \int_{\Omega} \left( \partial_F W( \nabla y(x) ) : \nabla z(x) + 2 | \sm( ( \nabla y(x) )^{\top} \nabla z(x) ) |^{2} \right)
\,\md x  \]
for any $ z \in W^{1,2}( \Omega ; \R^{3} ) $.
We immediately note that the quadratic form in the energy is not fully coercive, as $2 | \sm( ( \nabla y(x) )^{\top} \nabla z(x) ) |^{2}  = 0$ for skew-symmetric $( \nabla y(x) )^{\top} \nabla z(x)$. This is in a sense similar to the problem of linearized elasticity. It appears that the presence of the first term $ \partial_F W( \nabla y(x) ) : \nabla z(x)$ means that $\F_y$ is not even bounded from below for directions $\nabla z$ where the linearized dissipation vanishes. This, however, is not the case since the second Piola-Kirchhoff stress tensor $ ( \nabla y(x) )^{-1} \partial_F W( \nabla y(x) ) $ is always symmetric.

Let us for the sake of clarity denote the density function by
\begin{equation}
f_{Y}(Z) = \partial_F W( Y ) : Z + 2 | \sm( Y^{\top}  Z ) |^{2}.
\end{equation}
If $ Y^{-1} \partial_F W(Y) $ is symmetric, for any $ Z \in \R^{d \times d} $ it holds
\begin{equation}
\partial_F W(Y) \colon Z
= Y Y^{-1} \partial_F W(Y) \colon Z
= Y^{-1} \partial_F W(Y) \colon Y^{\top} Z
= Y^{-1} \partial_F W(Y) \colon \sm( Y^{\top} Z ).
\end{equation}
Therefore,
\begin{align*}
f_{Y}(Z)
& = \partial_F W(Y) \colon Z + 2 | \sm ( Y^{\top} Z ) |^{2} \\
& = Y^{-1} \partial_F W(Y) \colon \sm( Y^{\top} Z ) + 2 | \sm ( Y^{\top} Z ) |^{2} \\
& = 2 | \tfrac{1}{4} Y^{-1} \partial_F W(Y) + \sm( Y^{\top} Z ) |^{2} - \tfrac{1}{8} | Y^{-1} \partial_F W(Y) |^{2}.
\end{align*}
Hence, the minimum of $ f_{Y} $ is $ - \frac{1}{8} | Y^{-1} \partial_F W(Y) |^{2} $ and is attained at $Z$ for which it holds that  $ \sm( Y^{\top} Z ) = - \frac{1}{4} Y^{-1} \partial_F W(Y) $, i.e.~the set of minimizers is $ - \frac{1}{4} ( Y Y^{\top} )^{-1} \partial_F W(Y) + Y^{-\top} \Skew(d) $.
Moreover, it holds
\begin{equation*}
f_{Y}(Z) \ge  2 | \sm ( Y^{\top} Z ) |^{2} - | Y^{-1} \partial_F W(Y) | \cdot | \sm( Y^{\top} Z ) |
\end{equation*}
Therefore, the density $ f_{Y} $ and, consequently, the functional $ \F_y $ are bounded from below.

Now we consider the lack of coercivity since the functional only depends on $ \sm( (\nabla y)^{\top} \nabla z) $. Here, we rely mainly on the results in \cite{Neff2002}. Let us start by determining the subspace
\[ \mathcal{N}_{y} = \{ z \in W^{1,2}( \Omega ; \R^{3} ) : \sm( (\nabla y)^{\top} \nabla z) = 0 \}. \]
First we notice
\[ \sm( ( \nabla y )^{\top} \nabla z ) = 0
\iff \sm(  \nabla z ( \nabla y )^{-1} ) = 0. \]
From Lemma 4.1 in \cite{Neff2002} a higher regularity follows, namely $ z \in C^{2}( \overline{\Omega} ; \R^{3} ) $ and $ \nabla z ( \nabla y )^{-1} \in C^{1,1/2}( \overline{\Omega} ; \R^{3 \times 3} ) $.
In Corollary 3.10 in \cite{Neff2002} it was shown that if we have matrix fields $ A , B \in C^{1}( \overline{\Omega} ; \R^{3 \times 3} ) $ such that $A$ is skew-symmetric and $ B = \nabla \psi $, then
\[ \curl(AB) = 0 \ \Longrightarrow \ A \mbox{ is constant}. \]
For $ z \in \mathcal{N}_{y} $ the matrix field $ A = \nabla z ( \nabla y )^{-1} $ is skew-symmetric. If we take $ B = \nabla y $, all the assumptions are fulfilled since
\[ \curl(AB)
= \curl( \nabla z ( \nabla y )^{-1} \nabla y )
= \curl \nabla z
= 0. \]
Hence, $ A $ is constant and thus $ \nabla z = A \nabla y $. Consequently, $ z = a + A y $ for some vector $ a \in \R^{3} $. Equivalently,
$ z = a + \omega \times y $ for some vectors $ a , \omega \in \R^{3} $.
Thus, we have shown
\[ \mathcal{N}_{y} = \{ z \in W^{1,2}( \Omega ; \R^{3} ) : \sm( (\nabla y)^{\top} \nabla z) = 0 \}
= \R^{3} + \Skew(3) y. \]
It is a 6-dimensional subspace.

\begin{theorem}
\label{theo:lin-functional-min}
For a given $ y\in C^{2}(\overline{\Omega};\R^{3})$, let us consider the functional
\[
\F_y(z) = \int_{\Omega}  \left( \partial_F W( \nabla y(x) ) : \nabla z(x) + 2 | \sm( ( \nabla y(x) )^{\top} \nabla z(x) ) |^{2} \right)
\,\md x.
\]
\begin{enumerate}
    \item[(a)] $ \F_{y} $ admits an unique minimizer on any subspace $ H \subset W^{1,2}(\Omega) $ with $ H \cap \mathcal{N}_{y} = \{ 0 \} $. This includes the case $ H = \{ z \in W^{1,2}( \Omega ) : z|_{\Gamma_{D}} = 0 \} $ for some smooth $ \Gamma_{D} \subset \partial \Omega $ with $ \mathcal{H}^{d-1}( \Gamma_{D} ) > 0 $.
    \item[(b)] On the whole $ W^{1,2}( \Omega ) $ minimizers exists and are unique up to an addition of elements from $ \mathcal{N}_{y} $.
\end{enumerate}
\end{theorem}

The proof will rely on generalized Korn's second inequality. For the proof, see Corollary 4.6, 4.7 and Section 5 in \cite{Neff2002}.
\begin{lemma}[Generalized Korn's second inequality]
Let $ \Omega \subset \R^{3} $ be a Lipschitz domain and $ y \in C^{1}( \overline{\Omega} ; \GL(3) ) $. Then
\[ z \mapsto \| \sm ( (\nabla y)^{\top} \nabla z ) \|_{L^{2}(\Omega)} + \| z \|_{L^{2}(\Omega)} \]
is a norm on $ W^{1,2}( \Omega ; \R^{3} ) $ that is equivalent to $ \|\cdot\|_{W^{1,2}( \Omega )} $.
\end{lemma}
As for the standard Korn's inequality, one proves that for any subspace $ H \subset W^{1,2}( \Omega ; \R^{3} ) $ with $ H \cap \Skew(3) y = \{0\} $ there exists a constant $C$ such that for every $ z \in H $
\[ \| \sm( (\nabla y)^{\top} \nabla z ) \|_{ L^{2}( \Omega ) } \ge C \| \nabla z \|_{ L^{2}( \Omega ) }. \]

\begin{proof}[Proof of Theorem~\ref{theo:lin-functional-min}]
First, let us consider a subspace $H$ with $ H \cap \mathcal{N}_{y} = \{ 0 \} $. The functional $ \F_y $ is strictly convex on $H$.
This can be easily seen since for a given $Y \in \R^{d \times d} $ the function
\[ f_{Y}(Z) = Y^{-1} \partial_F W(Y) \colon \sm( Y^{\top} Z ) + 2 | \sm ( Y^{\top} Z ) |^{2} \]
is a composition of a linear function $ Z \mapsto \sm( Y^{\top} Z ) $ and a strictly convex function
$ X \mapsto Y^{-1} \partial_F W(Y) \colon X + 2 |X|^2 $.
Moreover, as
\[ \F_y(z) \ge C \| \sm( (\nabla y)^{\top} \nabla z ) \|_{ L^{2}( \Omega ) } - C \ge C \| \nabla z \|_{ L^{2}( \Omega ) } - C, \]
there is a unique minimizer $ z_{H} $ on $H$.

See Theorem 4.3 in \cite{Neff2002} for the proof of
$ \{ z \in W^{1,2}( \Omega ) : z|_{\Gamma_{D}} = 0 \} \cap \mathcal{N}_{y} = \{0\} $.

For the whole space case, let $ H = \mathcal{N}_{y}^{\perp} $ be the orthogonal complement of $ \mathcal{N}_{y} $ in $ W^{1,2}( \Omega ; \R^{3} ) $.
By (a) there is a unique minimizer $ z_{\mathcal{N}_{y}^{\perp}} $ on $\mathcal{N}_{y}^{\perp}$. Therefore, $ \F_y $ admits a minimizer on $ W^{1,2}( \Omega ; \R^{3} ) $ and the set of all minimizers reads
\[ z_{\mathcal{N}_{y}^{\perp}} + \mathcal{N}_{y}. \]
\end{proof}
\begin{remark}
Since by \eqref{eq:R-concrete}
\[
\F_y(z) = \int_{\Omega}  \Big( \partial_F W( \nabla y(x) ) : \nabla z(x) + R ( \nabla y(x) , \nabla z(x) ) \Big)
\,\md x,
\]
the Euler-Lagrange equation for the functional $ \F_y $ is
\begin{equation}
    0 = - {\rm div} \Big( \partial_{F} W( \nabla y ) + \partial_{ \dot{F} } R ( \nabla y(x) , \nabla z(x) ) \Big).
\end{equation}
or
\begin{equation}
\label{eq:E-L-F-u}
0 = - {\rm div} \Big( \partial_{F} W( \nabla y ) + 2  \nabla y ( \nabla z )^{\top}  \nabla y + 2  \nabla y ( \nabla y )^{\top}  \nabla z \Big).
\end{equation}
Let us return to \eqref{eq:step} and suppose that $z$ is a minimizer of $ \F_{y^{k-1}} $. By plugging
$ z = (y^{k} - y^{k-1})/\tau $
in \eqref{eq:E-L-F-u}, we arrive at
\begin{equation} -{\rm div}\Big(\partial_F W( \nabla y^{k-1} ) + \partial_{\dot F} R(\nabla y^{k-1},  \frac{ \nabla y^{k} - \nabla y^{k-1} }{\tau} )  \Big) = 0.
\end{equation}
Comparing this to \eqref{eq:PDEvisco-formulation}
\begin{equation}
-{\rm div}\Big(\partial_FW(\nabla y^k_\tau)   + \partial_{\dot F}R(\nabla y_\tau^k, \frac{\nabla y^k_\tau-\nabla y^{k-1}_\tau}{\tau})  \Big) = 0,
\end{equation}
we notice that in the limit $ \tau \to 0 $, the new iteration scheme should yield the right solution -- since both are discretizations of the evolution equation \eqref{eq:viscoel}.
\end{remark}

\section{Numerical results} \label{sec:numerics}
We now consider a numerical implementation of the time-step method given in equation \eqref{eq:step}, where the velocity $z$ is computed in each time-step by minimizing \eqref{eq:gam_limit_min_prob}. For simplicity, we denote the time-step size here again by $\tau$. Our implementation is based on finite elements, so we approximate
\[
y^{k-1} \approx y_{h}^{k-1} = \sum_j \left(\mathbf{y}^{k-1}\right)_j \phi_j
\]
with a vector $\mathbf{y}^{k-1}$ of degrees of freedom and finite element basis functions $\phi_j$.
In each time-step we thus solve a linear system of the form
\[
K\mathbf{z} = \mathbf{b},
\]
where the symmetric system matrix is given by
\begin{equation} \label{eq:stiffness}
K_{ij} = \int_{\Omega} c\left((\nabla \phi_i)^\top \nabla y_{h}^{k-1} + (\nabla y_{h}^{k-1})^\top \nabla \phi_i\right) :  \left((\nabla \phi_j)^\top \nabla y_{h}^{k-1} + (\nabla y_{h}^{k-1})^\top \nabla \phi_j\right) \,\md x
\end{equation}
corresponding to a nonlinear dissipation potential $D(F,G) = c|F^\top F - G^\top G|$, compared to the exposition above, we introduced a scaling parameter $c>0$. The right hand side vector is
\[
\mathbf{b}_i = -\int_\Omega \left(\partial_F W(y_{h}^{k-1})\right) : \nabla \phi_i \,\md x.
\]
The linear system is solved using the sparse direct Cholesky solver CHOLMOD \cite{Cholmod} after a small shift of $K$ by $\delta M$ (with finite element mass matrix $M$) to ensure positive definiteness.

Finally we compute
\begin{equation} \label{eq:discrete_timestep}
\mathbf{y}^{k} = \mathbf{y}^{k-1} + \tau \mathbf{z}
\end{equation}
and iterate this procedure.
The finite element system can be endowed with appropriate boundary conditions, body forces, and surface tractions in the usual way. In the following we use the Neo-Hookean material model
\begin{equation} \label{eq:neohook}
W(F) = \frac{\mu}{2} \left(|F|^2-3-2\log\det F\right) + \frac{\lambda}{2} \left(\det F -1\right)^2
\end{equation}
with Lamé parameters $\mu$ and $\lambda$. All 3d simulations are performed using P2 finite elements on a tetrahedral mesh.

\paragraph{Experiment 1: a viscoelastic plate.}

\newcommand\factor{0.95}
\begin{figure}
\begin{subfigure}{0.3\textwidth} \centering
  \includegraphics[width=\factor\textwidth,trim={0 0 0 2.6cm},clip=true]{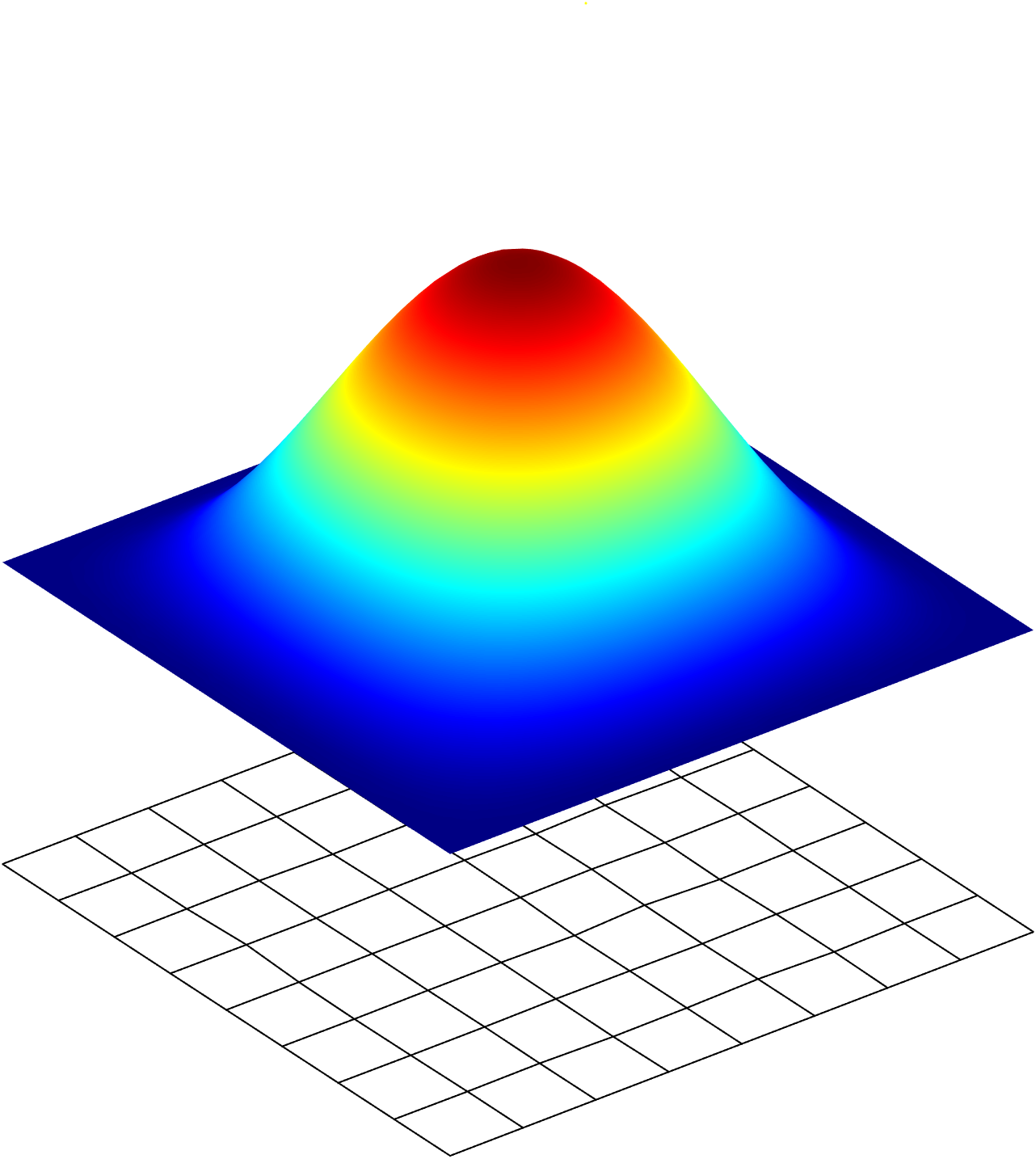} \\[5mm]
  \includegraphics[width=\factor\textwidth,trim={0 0 0 2.6cm},clip=true]{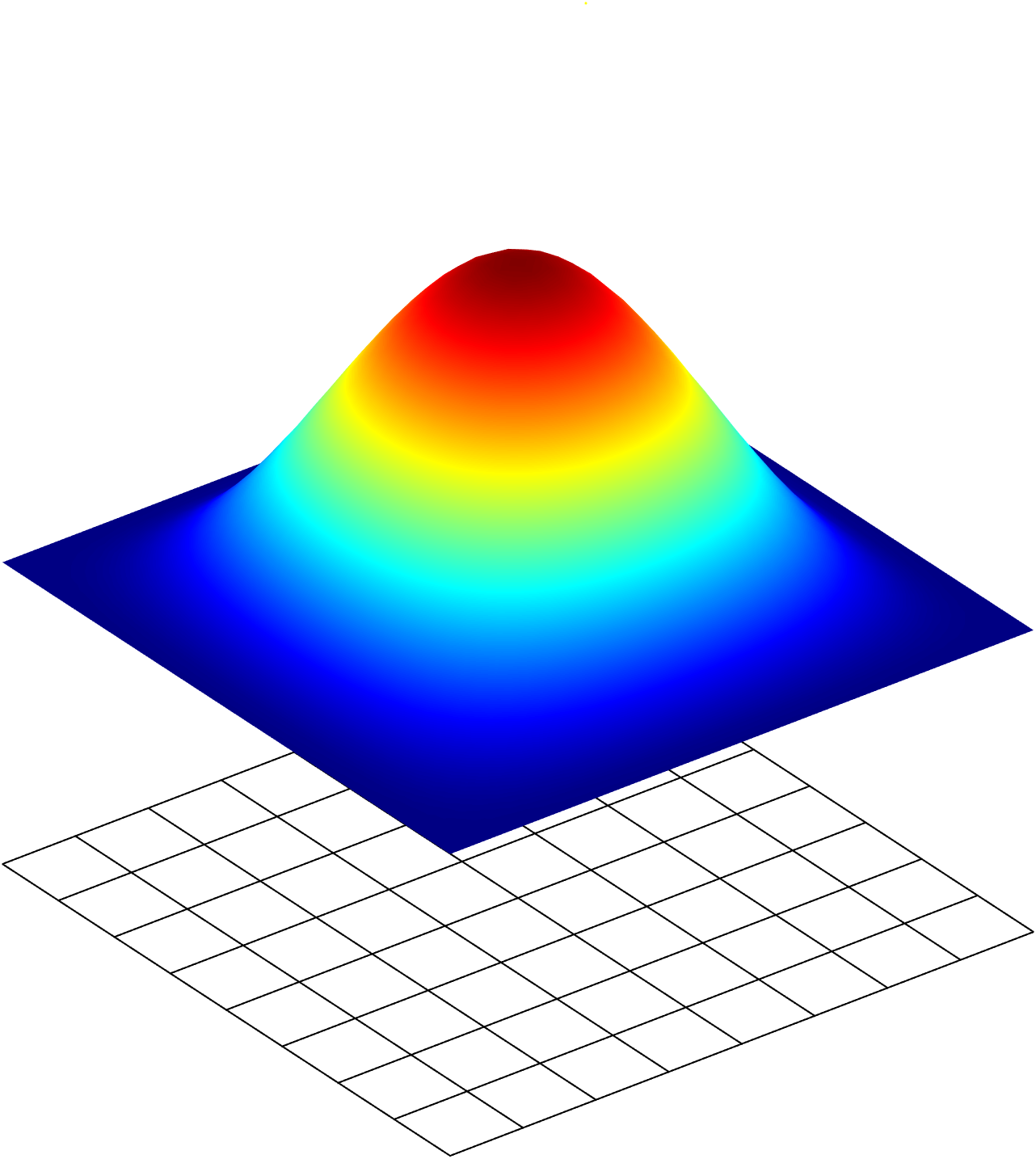}
  \caption{$t=0$.} \end{subfigure}
  \begin{subfigure}{0.3\textwidth} \centering
  \includegraphics[width=\factor\textwidth,trim={0 0 0 2.6cm},clip=true]{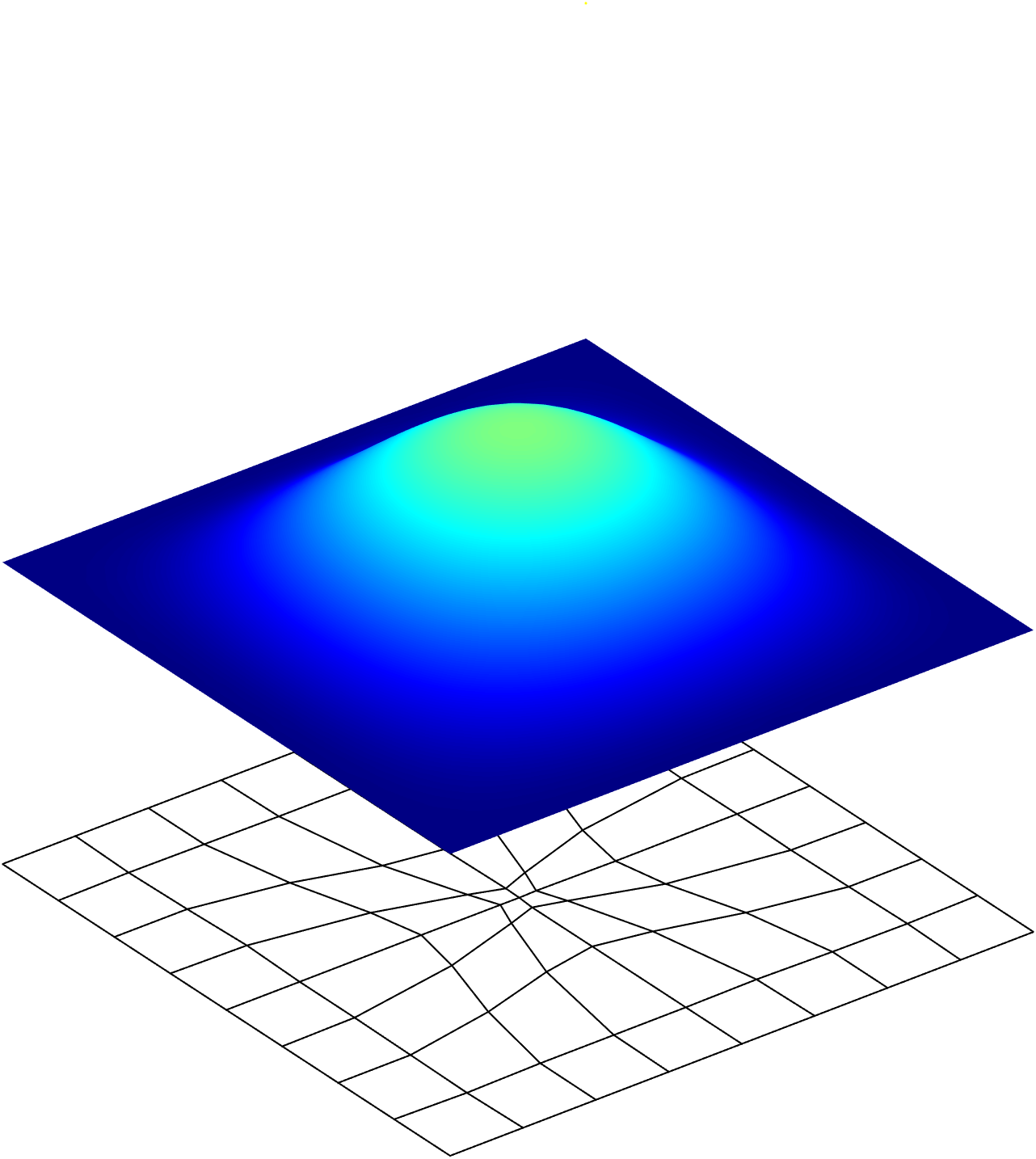} \\[5mm]
  \includegraphics[width=\factor\textwidth,trim={0 0 0 2.6cm},clip=true]{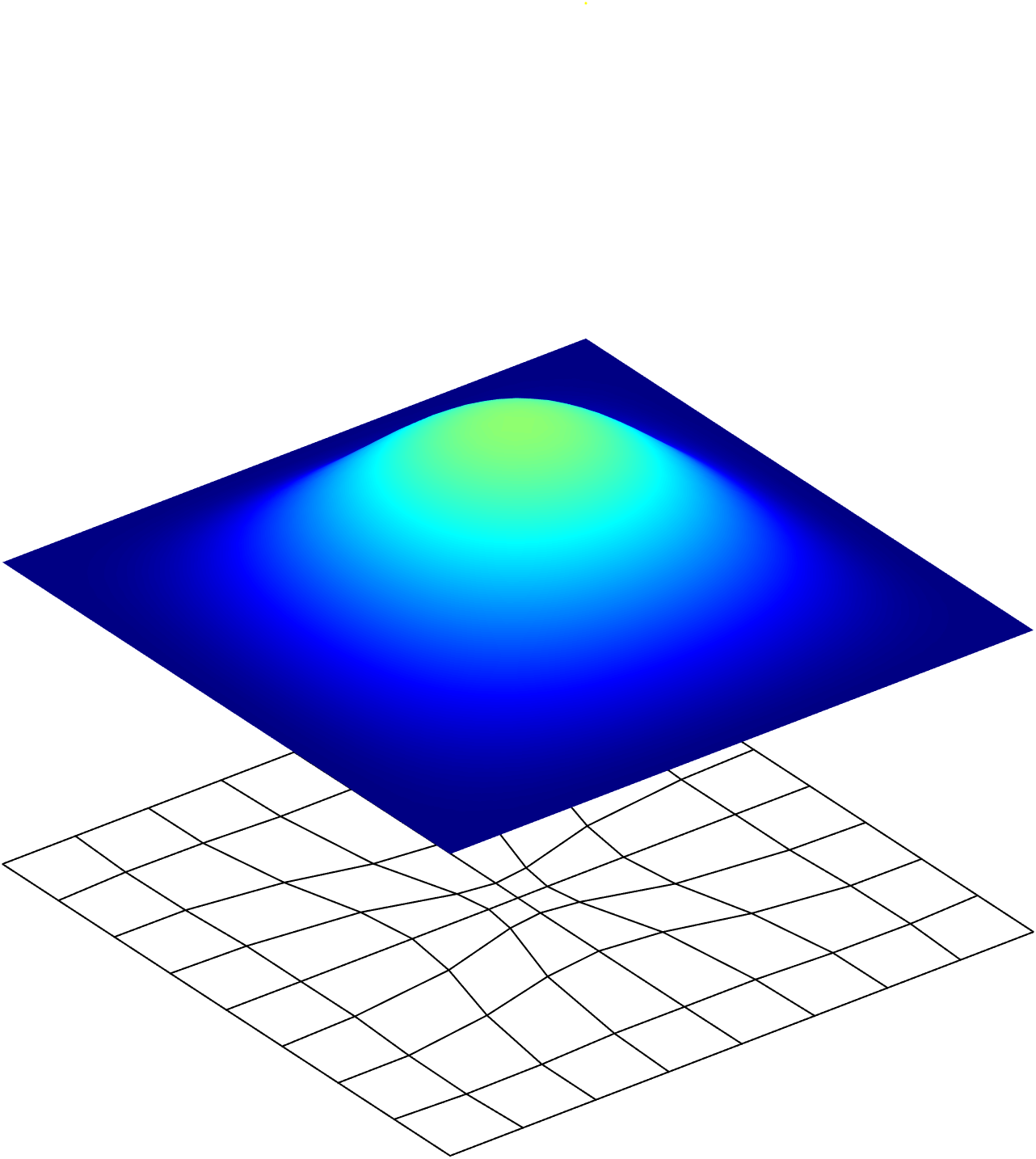}
  \caption{$t=4$.} \end{subfigure}
  \begin{subfigure}{0.3\textwidth} \centering
  \includegraphics[width=\factor\textwidth,trim={0 0 0 2.6cm},clip=true]{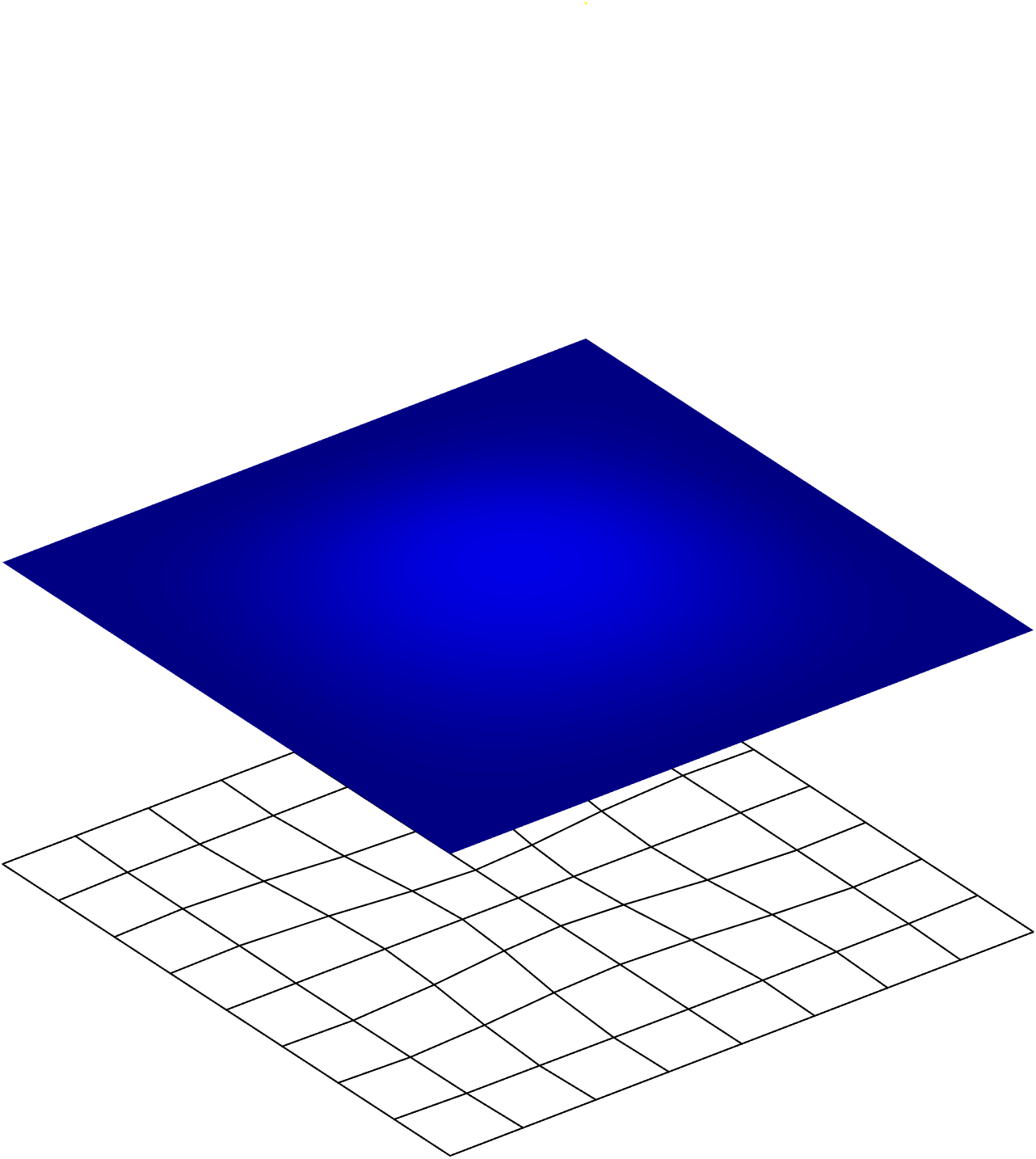} \\[5mm]
  \includegraphics[width=\factor\textwidth,trim={0 0 0 2.6cm},clip=true]{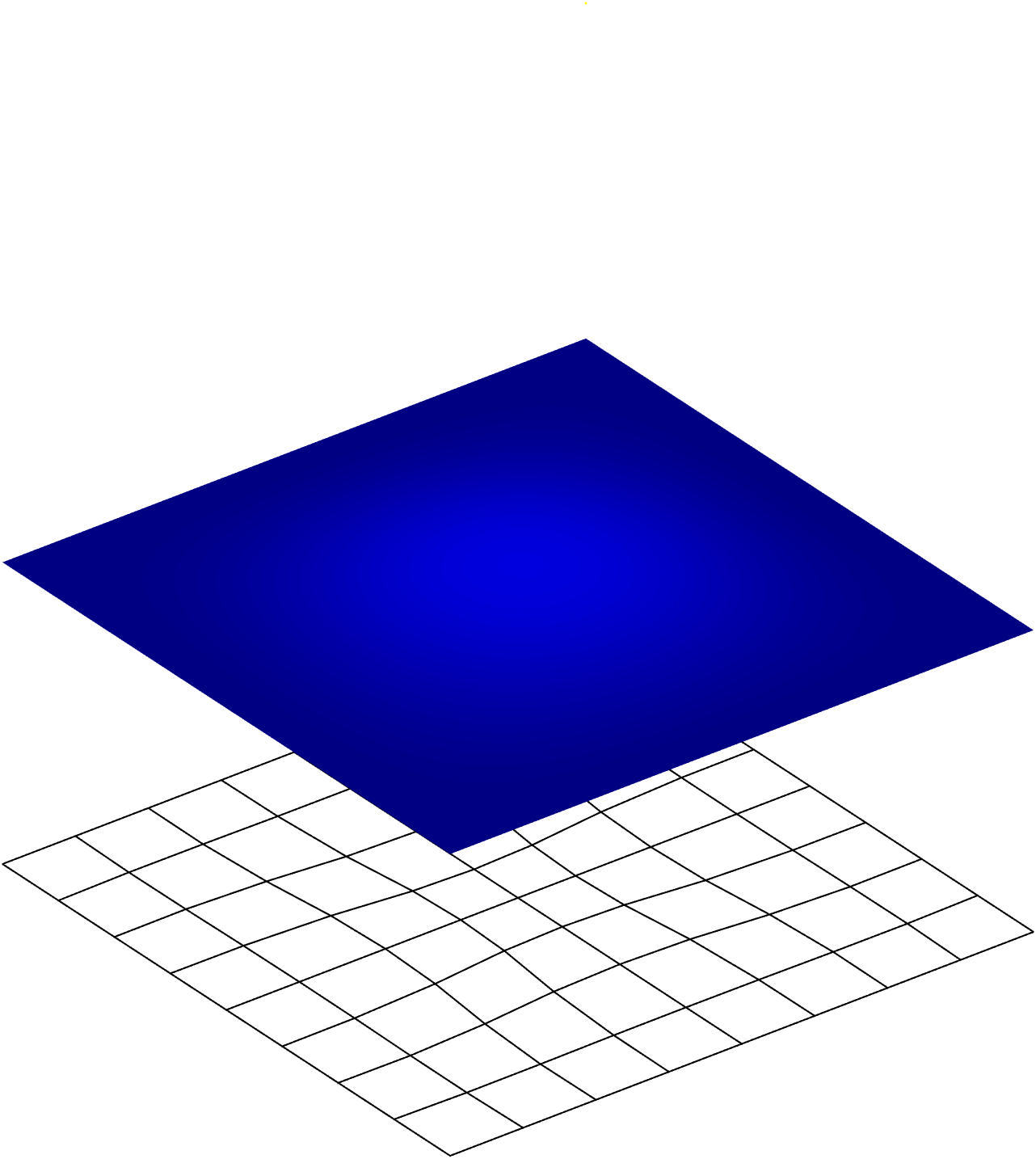}
  \caption{$t=13$.} \end{subfigure}
  \begin{subfigure}{0.045\textwidth} \centering
  \includegraphics[width=\factor\linewidth]{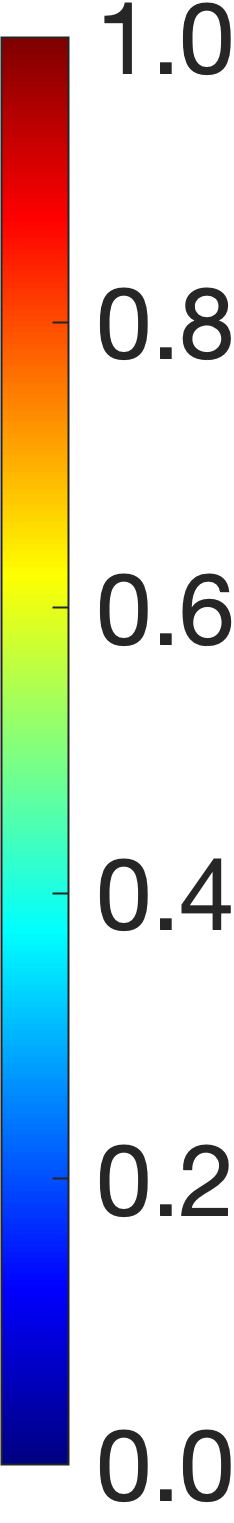}
  \end{subfigure}

\caption{Comparison of the simulation results using the 3d viscoelasticity method introduced here (top row, shown are out-of-plane ($v$) and in-plane ($u$) deformation of the mid-plane of the 3d plate), and the results using the limiting viscoelastic Föppl–von Kármán plate using the method from \cite{mfmkjv} (bottom row). The in-plane deformation is exaggerated by a factor of 6.}
\end{figure}

\begin{figure}
\begin{subfigure}[t]{0.49\textwidth} \centering
\includegraphics[width=\factor\textwidth]{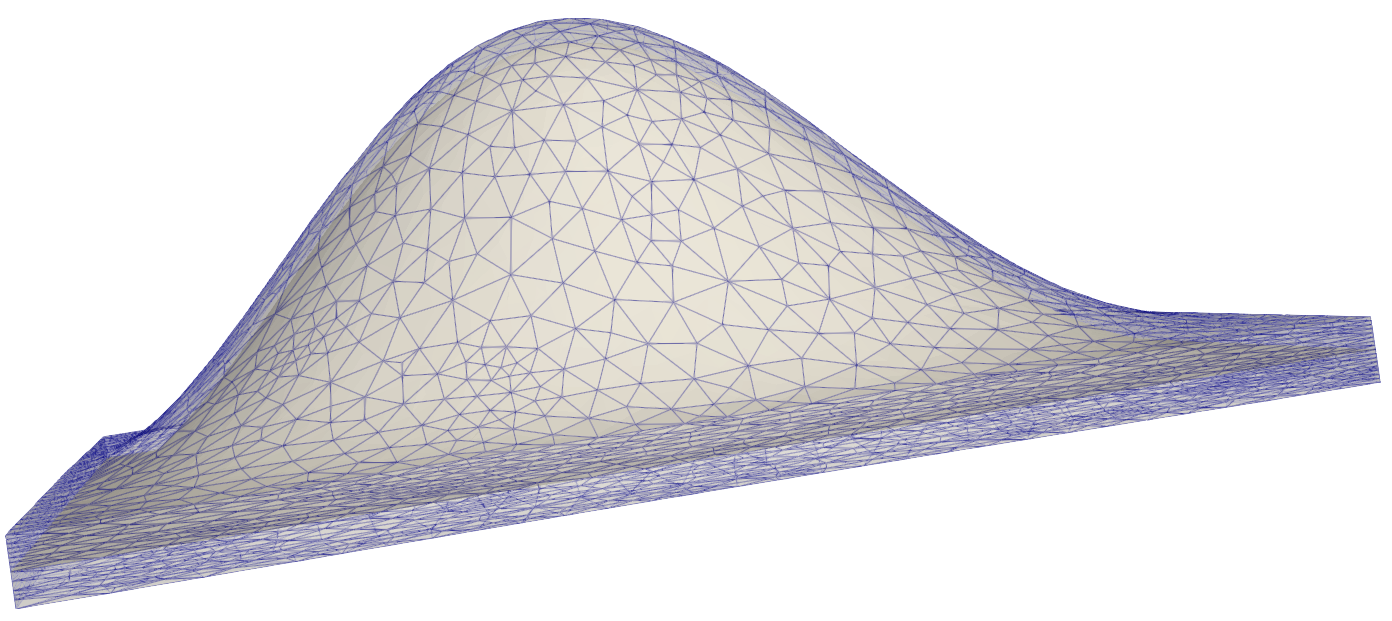}
\caption{3d deformation at $t=2$. The out of plane deformation has been rescaled according to the Föppl–von Kármán scaling. The mid-plane (which is also seen in Figure \ref{fig:plate_sim1}) is plotted opaque, the remaining material is transparent. The surface triangles of the tetrahdral mesh are shown in blue.}
\end{subfigure}\;\;\;\;\;\;
\begin{subfigure}[t]{0.49\textwidth} \centering
\includegraphics[width=\factor\textwidth]{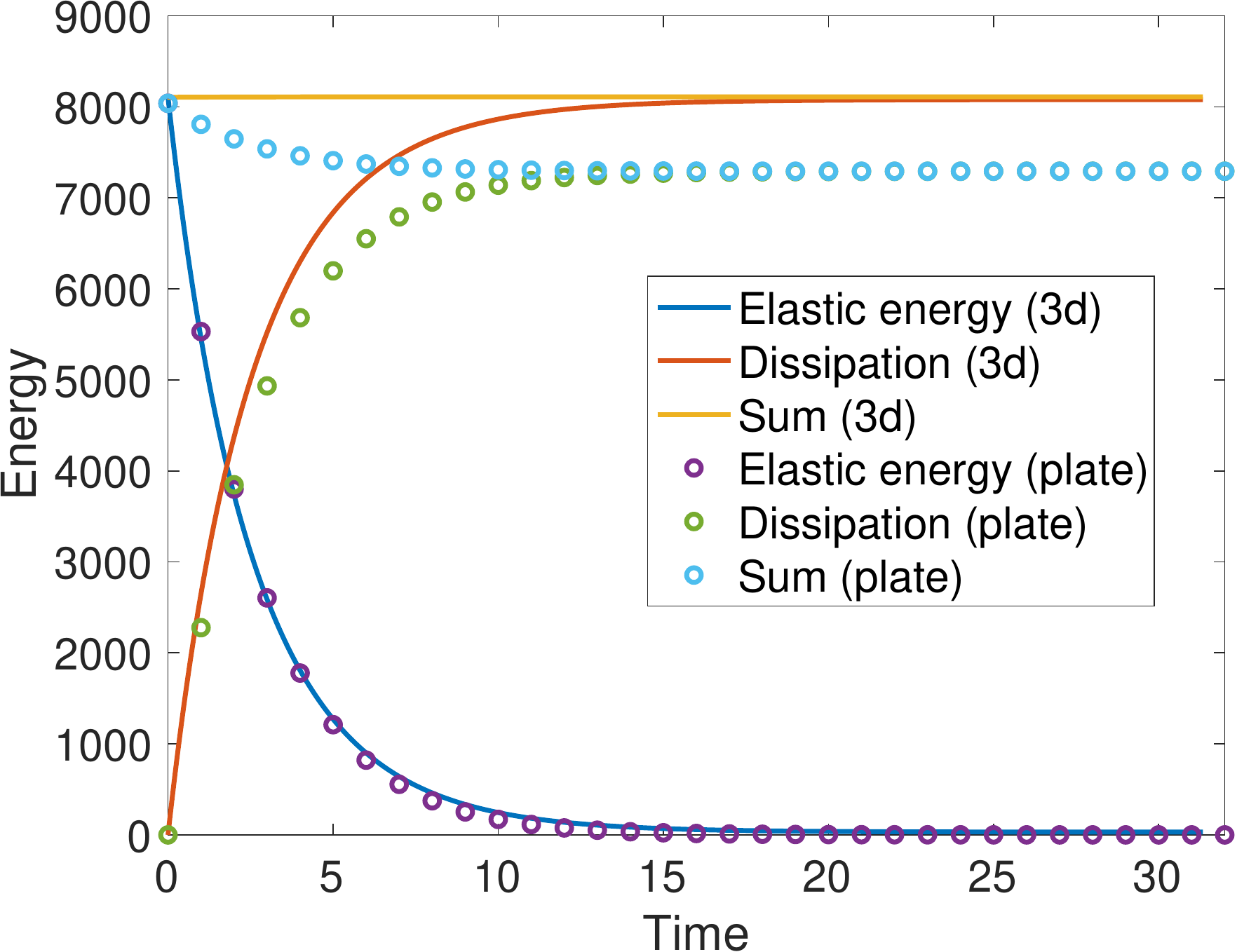}
\caption{The potential energy, the cumulative dissipation, and their sum. The results from the 2d plate simulation are shown as circles, the solid lines correspond to the 3d simulation.}
\end{subfigure}
\caption{The 3d deformation state and the time dependent elastic and dissipated energy.} \label{fig:plate_sim2}
\end{figure}

In \cite{mfmk} a rigorous thin plate limit for nonlinear viscoelastic materials was derived in the Föppl–von Kármán-scaling regime. Numerical experiments for the thin plate equations were performed in \cite{mfmkjv} combining  the Bogner-Fox-Schmit (BFS) finite elements \cite{BFS,Valdman} for the out-of-plane deformation $v\colon \omega \to \R$ and Q1 elements for the in-plane deformation $u\colon \omega\to \R^2$ on the plate domain $\omega = (-1,1)^2\subset \R^2$. We use this setting as a first test case for our method, computing the viscoelastic relaxation of an initially deformed plate with clamped boundary conditions. As initial condition we choose
\begin{align*}
v(x_1,x_2) &= (x_1-1)^2(x_2-1)^2, \\
u(x_1,x_2) &= 0.
\end{align*}
The thin plate model minimizes the energy $\phi(u,v) +\frac{1}{2\tau}\mathcal{D}^2( (u_{n-1},v_{n-1}),(u,v)),$
where the von K\'arm\'an energy functional $\phi$ is defined as
\begin{align}\label{eq: phi0}
{\phi}(u,v) := \int_S \frac{1}{2}Q_W^2\Big( e(u) + \frac{1}{2} \nabla v \otimes\nabla v \Big) + \frac{1}{24}Q_W^2(\nabla^2 v)\,\mathrm{d}x
\end{align}
and the global dissipation distance $\mathcal{D}$ as
\begin{multline}\label{eq: D,D0-2}
{\mathcal{D}}( (u_0,v_0),(u_1,v_1)) := \Big(\int_S Q^2_D\Big( e( u_1) - e(u_0) + \frac{1}{2} \nabla v_1 \otimes \nabla v_1 - \frac{1}{2} \nabla v_0 \otimes \nabla v_0 \Big) \\  +  \frac{1}{12}   Q_D^2\big(\nabla^2 v_1 - \nabla^2 v_0 \big)  \Big)^{1/2} \,\mathrm{d}x
\end{multline}
respectively.
As $Q^2_W$ and $Q^2_D$ we take the limiting Föppl–von Kármán plate material, i.e.,
$$Q^2_W(G ):= \frac{2\mu\lambda}{2\mu+\lambda} {\rm tr}^2(G)+ 2\mu|G|^2, \qquad Q^2_D(G):= 4c|G|^2\ , \ c>0\ $$
  for every  symmetric $G\in \R^{2\times 2}_{\rm sym}$. The constants $\lambda, \mu$ are the Lam\'e constants from \eqref{eq:neohook}  and $c>0$ is the same viscosity parameter as in \eqref{eq:stiffness}.

We then compare the solution of the thin-plate model to a full three-dimensional simulation using the algorithm from the beginning of Section \ref{sec:numerics} developed in this article, applied to a 3d plate with thickness $h$ in a reference domain $\omega \times (-\frac{1}{2},\frac{1}{2})$ and $x_3$ derivative rescaled by $\frac{1}{h}$. The initial configuration is given as the respective element of the recovery sequence for $v$, $u$ above as detailed in \cite[Section 6.2]{Friesecke.2006}, yielding approximately the same starting energy as the thin plate limit. The simulation was performed using 201\,514 tetrahedra resulting in 818\,538 degrees of freedom. The wall time until full relaxation at $t\approx 30$ was approximately 6 days on 8 cores of an Intel Xeon Gold 6230.

The parameters in this experiment are $\mu = \lambda = 1.0\cdot 10^3$, $c=3.0\cdot 10^3$ and $h=0.1$. The time step size for the plate simulation was $\tau_\textrm{plate} =1$, for the 3d-simulation it was $\tau_\textrm{3d} = 0.01$. A comparison of the deformation of the mid-plane in the 3d simulation with the limiting plate model is shown in Figure \ref{fig:plate_sim1}. Figure \ref{fig:plate_sim2} shows the 3d deformation at a $t=2$, as well as a comparison of the energy and cumulative dissipation of the plate model and the 3d simulation (see also the convergence study below). One can clearly see that our model and the plate limit yield very similar results -- as one would expect. We also point out the near perfect adherence to energy-dissipation-equality by the 3d simulations using the method developed here.

\paragraph{Experiment 2: Jello.}
\begin{figure}
\begin{subfigure}{0.32\textwidth} \centering
  \includegraphics[width=\factor\textwidth,trim={9cm 1.2cm 0 0},clip=true]{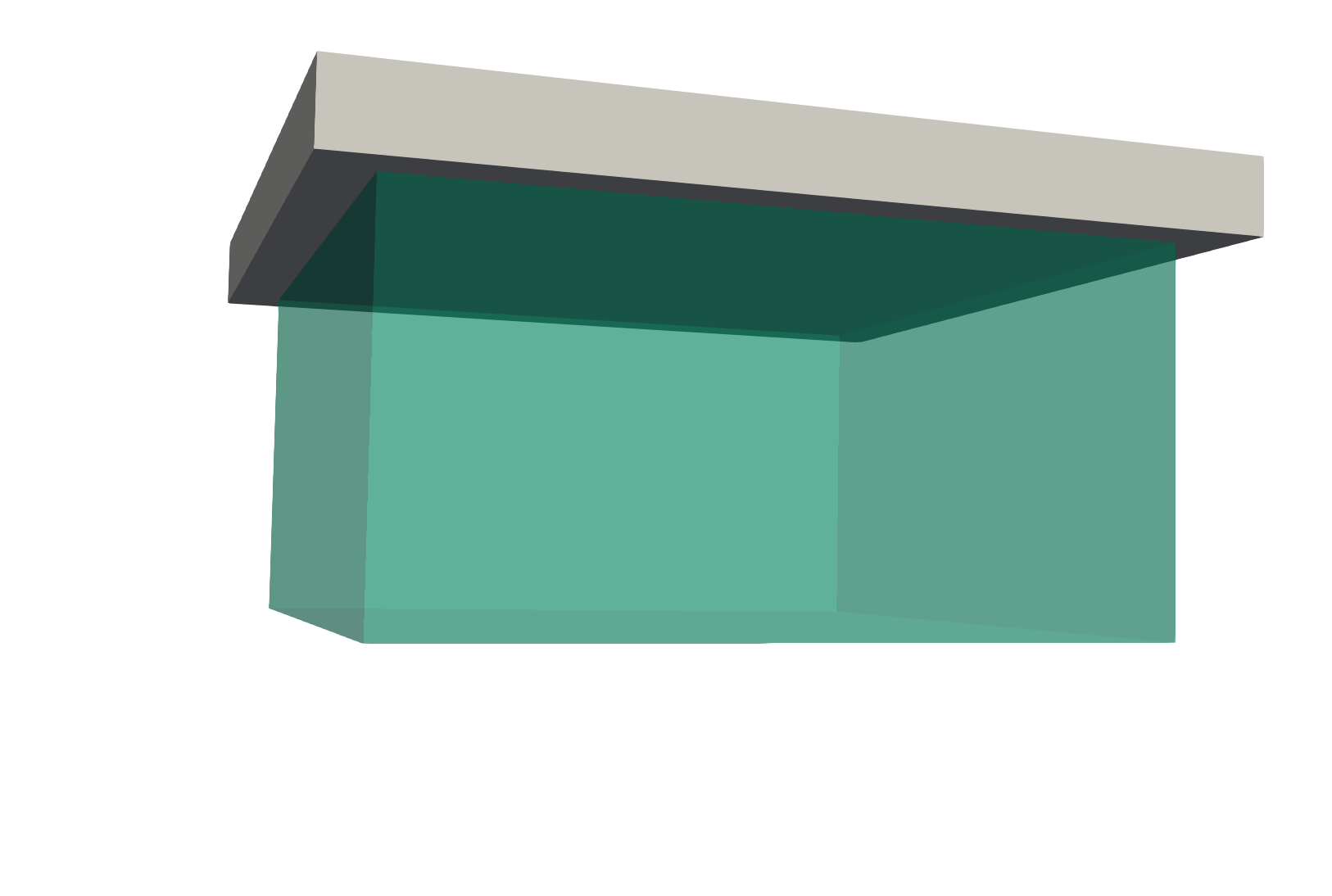}
  \caption{$t=0$.} \end{subfigure}
  \begin{subfigure}{0.32\textwidth} \centering
  \includegraphics[width=\factor\textwidth,trim={9cm 1.2cm 0 0},clip=true]{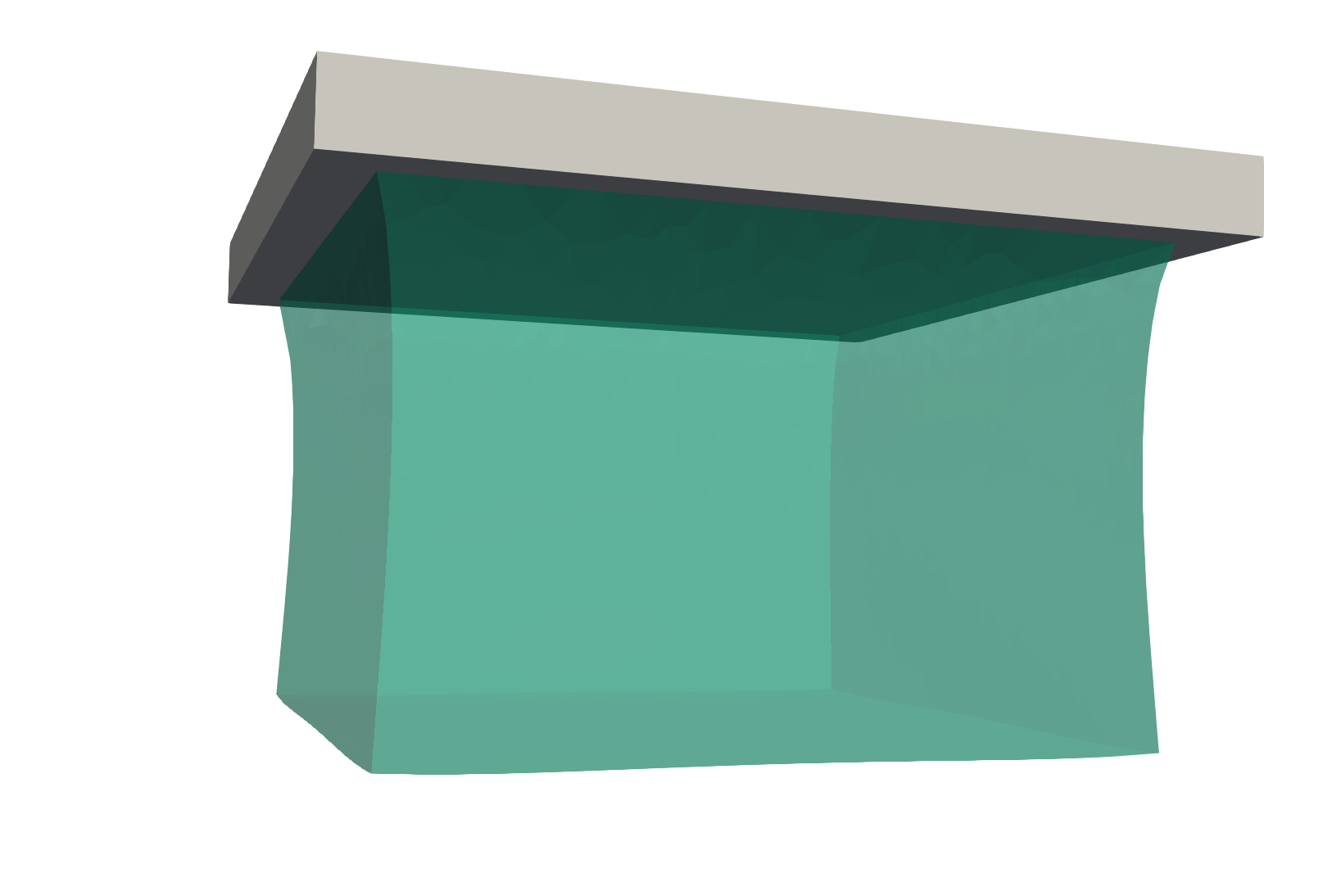}
  \caption{$t=10$.} \end{subfigure}
  \begin{subfigure}{0.32\textwidth} \centering
  \includegraphics[width=\factor\textwidth,trim={9cm 1.2cm 0 0},clip=true]{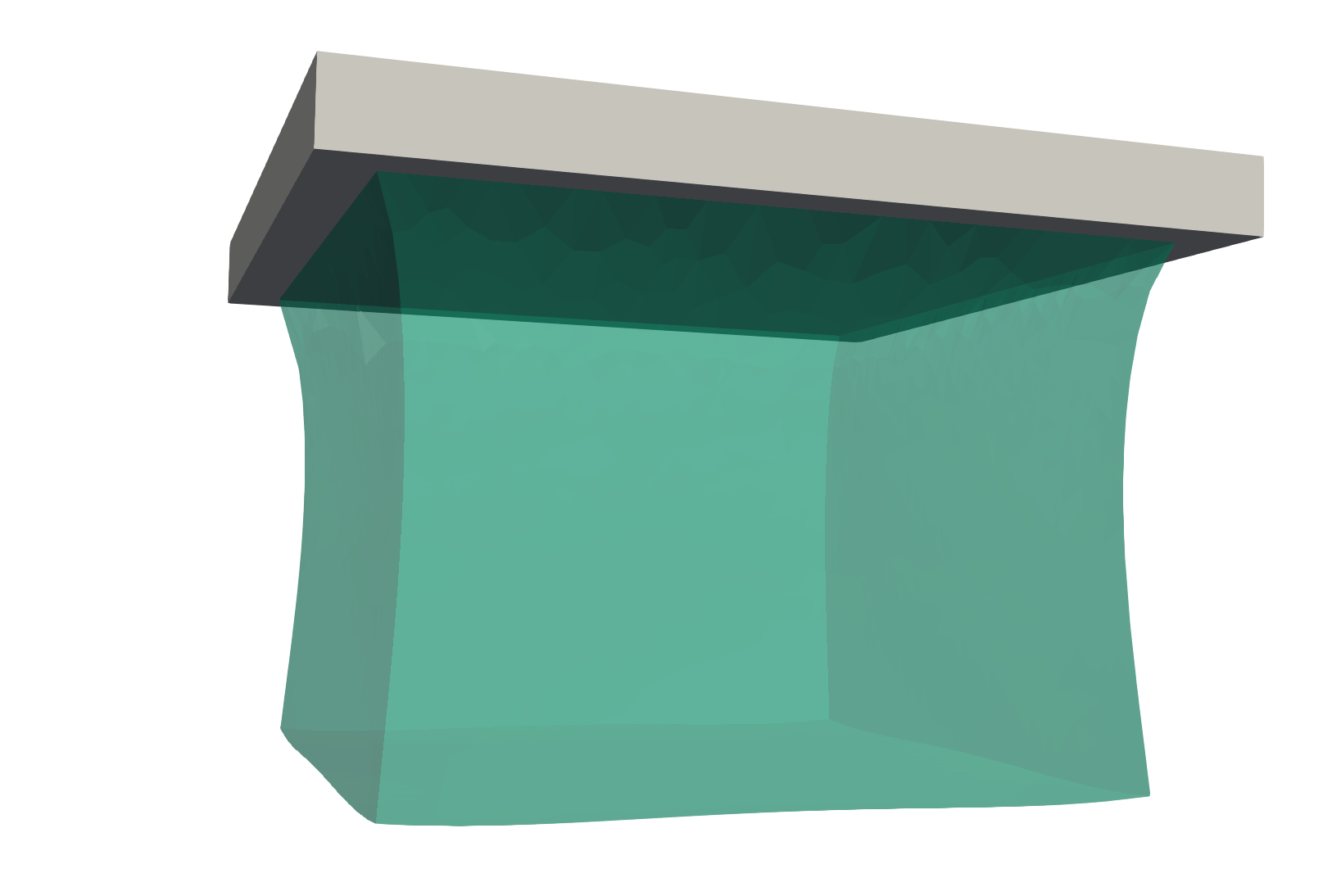}
  \caption{$t=30$.} \end{subfigure}
\caption{Simulation results for the viscoelastic (jello) material. The plate at the top indicates where the specimen is fixated, a gravitational body force acts downwards.} \label{fig:plate_sim1}
\end{figure}

We simulate the viscous evolution of a rectangular block of a material attached at the top under the influence of a gravitational body force. The jello has a reference configuration occupying a domain $\Omega = (-1,1) \times (-1,1) \times (-\frac{1}{2}, \frac{1}{2})$ and initial condition  $y(x)=x$. A Dirichlet boundary condition fixes $y|_{\{x_3 = \frac{1}{2}\}} = x$ for all time. The energy is a sum of the
\begin{wrapfigure}[14]{r}{0.4\linewidth}
\centering\includegraphics[width=0.9\linewidth]{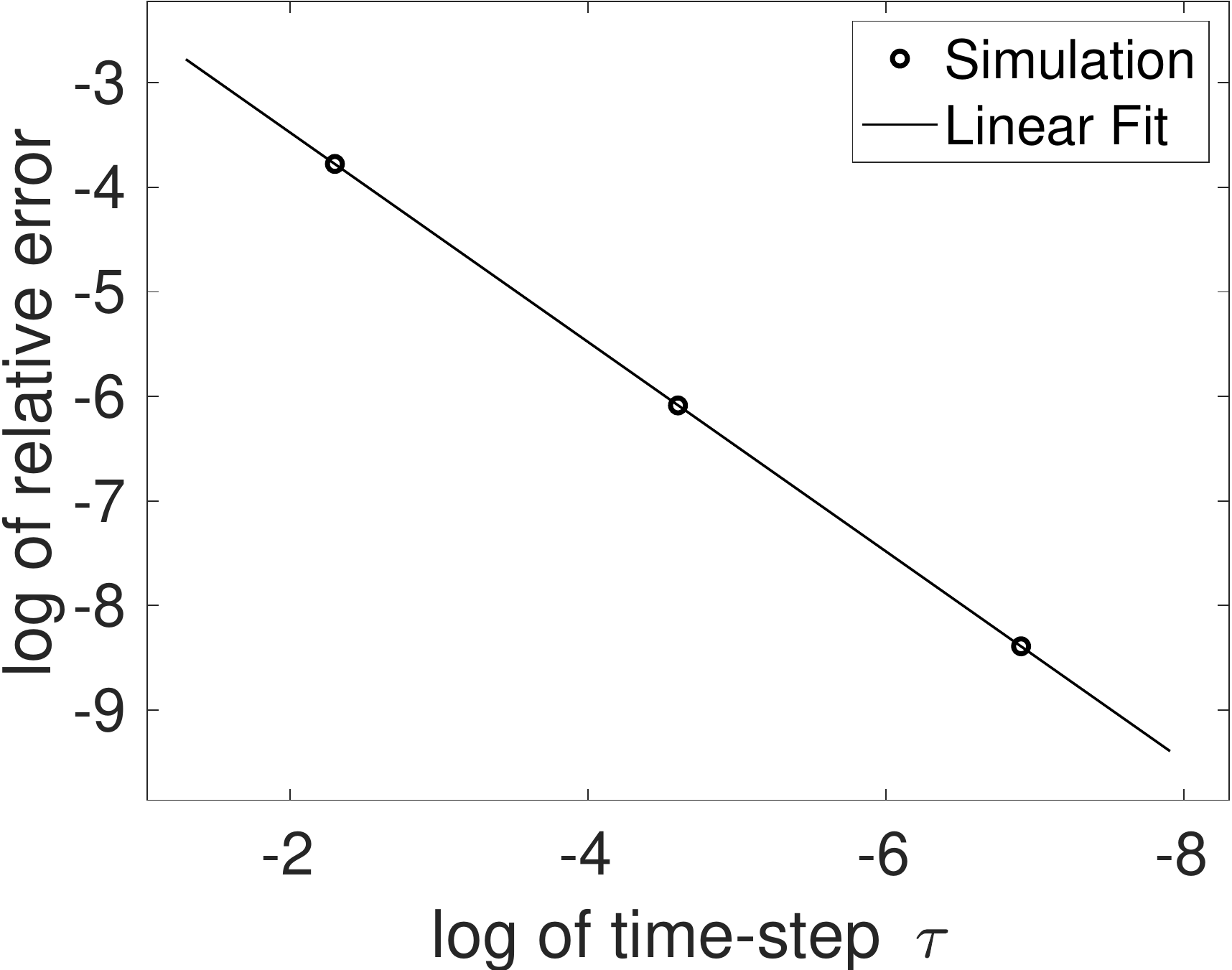}
\caption{Simulation using 147\,342 P2 finite elements. The slope of the linear fit is 1.00.} \label{fig:convergence}
\end{wrapfigure}
Neo-Hookean energy in \eqref{eq:neohook} and $-\int_\Omega f y_3(x)  \,\mathrm{d}x$. The parameters in this simulation are $\mu = 1.0\cdot 10^3$, $\lambda = 1.5\cdot 10^3$, $f=-2.0 \cdot 10^3$, $c=3.0\cdot 10^3$, and $\tau = 0.01$. Again, we see nearly perfect adherence to the energy-dissipation balance in this simulation. The number of degrees of freedom and the running time are the same as for experiment 1 above.

\paragraph{Convergence study.}

We perform a simple study testing the relative error in the energy-dissipation-equality depending on the time-step size $\tau$. On the reference domain $\Omega = (-1,1)^3$, we choose an initial condition
\[
y^0 = \begin{pmatrix}1.2 &0& 0 \\ 0& 1& 0 \\ 0& 0& 1 \end{pmatrix}x
\]
and perform the \emph{linearized} time-stepping scheme in \eqref{eq:discrete_timestep} until final time $T=1$, i.e., up to step $K_\tau=\frac{1}{\tau}$ for time-steps $\tau \in \{0.1, 0.01, 0.001\}$. Finally, we compute the relative error in the energy-dissipation equality with the  \emph{nonlinear} dissipation distance, i.e., we compute
\[
e_\tau = \frac{\left| W(y^0)-W(y^{K_\tau}) - \sum_{j=1}^{K_\tau} D(\nabla y^{j}, \nabla y^{j-1})\right|}{\sum_{j=1}^{K_\tau} D(\nabla y^{j}, \nabla y^{j-1})}.
\]
The simulation was performed using the parameters $\mu = 1.0\cdot 10^3$, $\lambda = 1.5\cdot 10^3$, $c=3.0\cdot 10^3$, no body force and stress-free boundary conditions on a domain discretized using 147\,342 P2 finite elements. From Figure \ref{fig:convergence} one can clearly obtain that $e_\tau = \mathcal{O}(\tau)$. The same scaling is also obtained using coarser discretizations with 71\,369 and 33\,216 P2 finite elements. We thus conclude that our linearized time-stepping scheme recovers the correct energy-dissipation equality for nonlinear Kelvin-Voigt viscoelasticity.

\section*{Acknowledgements}
PD and MJ are grateful for the hospitality afforded by The Institute of Information Theory and Automation of the Czech Academy of Sciences.
PD acknowledges support by the state of Baden-Württemberg through bwHPC and from the DFG via project 441523275 in SPP 2256. MK acknowledges the support and hospitality of the University of Freiburg during his stay there in January 2020.
MK and JV were supported by the GA\v{C}R project 21-06569K.

\bibliographystyle{abbrv}
\bibliography{DJKV_ONL}

\end{document}